\documentclass[12pt]{article}
\usepackage[a4paper,width=6.8in,height=8.9in]{geometry}
\usepackage{amsmath,amssymb}

\usepackage{amsmath,amssymb,mathtools,amsthm,
	textcomp,mathscinet}
\usepackage{amsfonts,graphicx,enumerate,subcaption,placeins}
\usepackage[hidelinks]{hyperref}
\usepackage{setspace}
\usepackage[utf8]{inputenc}
\usepackage[english]{babel}
\usepackage{booktabs,multirow}
\usepackage{blkarray}
\usepackage{lscape}
\usepackage[T1]{fontenc}
\usepackage{authblk,color}
\usepackage[english]{babel}
\usepackage{amsfonts, amstext, amsthm, booktabs, dcolumn}
\usepackage{graphicx}
\usepackage{float}
\usepackage{caption}
\usepackage{centernot}
\usepackage{times}
\usepackage[nottoc]{tocbibind}

\definecolor{gr}{rgb}{0.7, 0.0, 0.49}

\newtheorem{theorem}{\bf Theorem}[section]

\newtheorem{corollary}{\bf Corollary}[section]

\newtheorem{example}{\bf Example}[section]

\numberwithin{equation}{section}
\begin{document}
\title{Poisson Approximation to the Convolution of Power Series Distributions}
\author[$a$ ]{{\Large A. N. Kumar}}
\author[$a$ ]{{\Large P. Vellaisamy}}
\author[$b$ ]{{\Large F. Viens}}
\affil[$a$ ]{ Department of Mathematics, Indian Institute of Technology Bombay}
\affil[ ]{Powai, Mumbai-400076 India.}
\affil[ ]{and}
\affil[$b$ ]{Department of Statistics and Probability}
\affil[ ]{Michigan State University}
\affil[ ]{East Lansing, MI 48824, USA.}
\affil[ ]{Emails: amit.kumar2703@gmail.com; pv@math.iitb.ac.in; viens@msu.edu}

\date{}

\maketitle

\begin{abstract}
\noindent 
In this article, we obtain, for the total variance distance, the error bounds between Poisson and convolution 
of power series distributions via Stein's method. This provides a unified approach to many known discrete distributions. Several Poisson
limit theorems follow as corollaries 
from our bounds. As applications, we compare  the Poisson approximation results with the negative binomial 
approximation results, for the sums of Bernoulli, geometric, and logarithmic series random variables.
\end{abstract}

\noindent
\begin{keywords}
Convolution of distributions; Poisson and negative binomial approximation; power series distribution; Stein's method.
\end{keywords}\\
{\bf MSC 2010 Subject Classifications:} Primary: 62E17, 62E20; Secondary: 60F05, 60E05.

\section{Introduction} The convolution of distributions play an important role in several applications in the areas related to rare events, waiting time, and wireless communications, among many others. It is in general difficult to find the distributions of the sums of especially independent and non-identical random variables (rvs). In such cases, approximations to a known distribution is useful in applications. For example, Poisson approximation to the convolution of Bernoulli rvs is studied by Barbour and Hall \cite{BH}, Chen \cite{LC}, Kerstan \cite{KJ}, and Le cam \cite{CAM}. The Poisson approximation to the convolution of geometric rvs is studied by, for example, Barbour \cite{BAD}, Hung and Giang \cite{HLGT} and Teerapabolarn and Wongkasem \cite{TKWP}. The Poisson approximation to the convolution of negative binomial rvs is studied by Teerapabolarn \cite{Teeraa}, and Vellaisamy and Upadhye \cite{VPUNS}, among  others.

In this article, we focus on the convolution of power series distributions (PSD) and obtain the upper bounds for its approximation to Poisson distribution, using Stein's method. The distance metric used is the total variation distance. We show that the limit theorems given by P\'{e}rez-Abreu \cite{APV} follow from our results, as special cases. As examples, we discuss the Poisson convergence result for binomial, negative binomial and logarithmic series distributions. Furthermore, we mention the negative binomial approximation results and compare the bounds with Poisson approximation results, either  theoretically or numerically. It is shown that our bounds are either comparable to or an improvement over the existing bounds. We have discussed also the Poisson approximation to the convolution of logarithmic series distributions, which has not been studied so far in the literature.

The article is organized as follows. In Section \ref{Pre}, we discus some known results for PSD's and Stein's method. In Section \ref{PA}, we derive the error bounds between Poisson and the convolution of PSD's, and discuss some relevant consequences. In Section \ref{CPNB}, we present the results for negative binomial approximation to PSD obtained from Vellaisamy {\em et al.} \cite{VUC}. Finally, we give a numerical comparison between Poisson and negative binomial approximation to the PSD's. 

\section{Preliminaries}\label{Pre}
First we introduce the notation and briefly discuss the Stein's method. Let ${\mathbb Z}_+=\{0,1,2,\dotsc\}$, the set of non-negative integers, and $Z$ be a random variable (rv) with
\begin{equation}
{\mathbb P}(Z=k)=\frac{a_k \theta^k}{h(\theta)}, ~k \in {\mathbb Z}_+,~ \theta >0,
\end{equation}
where $a_k\ge 0$ and  $h(\theta)=\sum_{k=0}^{\infty}a_k \theta^k$. Then we say $Z$ belongs to the class of PSD's corresponding to $h(\theta)$. It can be easily verified that the Bernoulli, binomial, geometric, negative binomial and logarithmic series distributions, among many others, belong to the class of PSDs. For more details, we refer the reader to Johnson {\em et al.} \cite{JKK2005}. Through out the paper,  Poi($\lambda$) and Geo($p$) denote
respectively the Poisson and geometric distributions. \\
Next, we describe briefly Stein's method to obtain the error bound, under the total variation norm, between two discrete distributions. Stein's method involves mainly the following three steps: 
\begin{enumerate}
\item Let ${\cal G}=\{g:{\mathbb Z}_+\to{\mathbb R}|~g~\text{is bounded}\}$  be a class of bounded functions on $ {\mathbb R}.$ For a ${\mathbb Z}_+$-valued rv $Y$, define ${\cal G}_Y=\{g\in {\cal G}|~g(0)=0~\text{and}~g(x)=0,~\text{for}~x\notin \text{Supp}(Y)\}$. An operator ${\cal A}_Y$ defined by
\begin{equation}
{\mathbb E}[{\cal A}_Y g(Y)]=0,~\text{for}~g \in {\cal G}_Y,
\end{equation}
 is called a Stein operator for the rv $Y$.
\item Solve next equation
\begin{equation}
{\cal A}_Yg(k)=f(k)-{\mathbb E}f(Y),~f \in {\cal G}~\text{and}~g \in {\cal G}_Y,\label{seq}
\end{equation}
called the Stein equation.
\item Replace $k$ by a rv $Z$ in Stein equation \eqref{seq}, and take expectation and supremum to get
\begin{equation}
d_{TV}(Z,Y):=\sup_{f \in {\cal H}}|{\mathbb E}f(Z)-{\mathbb E}f(Y)|=\sup_{f \in {\cal H}}|{\mathbb E}{\cal A}_Yg(Z)|,\label{power}
\end{equation}
where ${\cal H}=\{I(A)| A \subseteq {\mathbb Z}\}$ and $I(A)$ is the indicator function of $A$ and 
$d_{TV}(Z,Y) $ is  the total variation distance
between the distributions of $Z$ and $Y.$ 
\end{enumerate}

\noindent Let now $X\sim$ Poi($\lambda$), the Poisson distribution, with probability mass function (pmf)
\begin{equation}
{\mathbb P}(X=k)=\frac{e^{-\lambda}\lambda^k}{k!},~k\in{\mathbb Z}_+,\label{popmf}
\end{equation}
for some $\lambda >0$. The Stein operator for $X$ is given by (see Barbour {\em et al.} \cite{BHJ})
\begin{equation}
{\cal A}g(k)=\lambda g(k+1)-k g(k), ~k\in{\mathbb Z}_+.\label{pooper}
\end{equation}
Also, the bounds for the solution of Stein equation \eqref{seq} are given by
\begin{equation}
\|g\|\le \frac{1}{\max(1,\sqrt{\lambda})}\quad\text{and}\quad\|\Delta g\|\le \frac{2 \|f\|}{\max(1,\lambda)}, \label{bound}
\end{equation}
where $\Delta g(k)=g(k+1)-g(k)$ and $\|\Delta g\|=\sup_{k}|g(k+1)-g(k)|$. Note that $2 \|f\|$ can be replaced by $1$ (see Upadhye {\em et al.} \cite{UCV}). For more details, we refer the reader to Barbour {\em et al.} \cite{BHJ} and Upadhye {\em et al.} \cite{UCV}.

\section{Poisson Approximation to PSD's}\label{PA}
We first consider the case of independent but non-identical PSD's. Let $X_{i,n}$, $1 \le i\le n$, $n\ge 1$, be a double array  of independent rvs having PSD with probability mass function (pmf)
\begin{align}
{\mathbb P}(X_{i,n}=k)=p_{i,n}(k)=\frac{a_k \theta_{i,n}^k}{h(\theta_{i,n})},~k\in{\mathbb Z}_+,\label{psdpmf}
\end{align}
where $a_k\ge 0$, $\theta_{i,n}>0$, and $h(\theta_{i,n})=\displaystyle{\sum_{k=0}^{\infty}}a_k \theta_{i,n}^k$, $1 \le i \le n$. We call, for simplicity, $X_{i,n}$ power series rvs and the distribution in \eqref{psdpmf} the PSD associated with the function $h$. We assume $h$ is differentiable. Note that
\begin{equation}
{\mathbb E}X_{i,n}=\sum_{k=1}^{\infty}kp_{i,n}(k)=\frac{1}{h(\theta_{i,n})}\sum_{k=1}^{\infty}k a_k\theta_{i,n}^k=\frac{\theta_{i,n}h^\prime(\theta_{i,n})}{h(\theta_{i,n})}.\label{psdmean}
\end{equation}
Since $h^\prime(\theta_{i,n})=\sum_{k=1}^{\infty}k a_k \theta_{i,n}^{k-1}=\sum_{k=0}^{\infty}(k+1)a_{k+1}\theta_{i,n}^k$, we have

\begin{equation}
 \sum_{k=0}^{\infty}\frac{(k+1)a_{k+1}\theta_{i,n}^k}{h^\prime(\theta_{i,n})}=1.\label{hprime}
\end{equation}
Let $X_{i,n}^*$, $1 \le i\le n$, $n\ge 1$, be a sequence of independent power series rvs corresponding to $h^\prime$ so that pmf of $X_{i,n}^*$ is given by
\begin{equation}
p_{i,n}^*(k)={\mathbb P}(X_{i,n}^*=k)=\frac{(k+1)a_{k+1}\theta_{i,n}^k}{h^\prime(\theta_{i,n})},~k\in {\mathbb Z}_+.\label{primepsdpmf}
\end{equation}
For $n \geq 1,$ let $S_n=\sum_{i=1}^{n}X_{i,n}, n \geq 1,$
denote the sequence of the partial sums of the independent power series rvs. Now, our main interest in this article is to study Poisson approximation to $S_n$, and obtain error bounds. Our results unifies several known results obtained for specific distributions, such as binomial and geometric, belonging to PSD's. Let $N_\lambda$ henceforth denote the Poisson rv with mean $\lambda>0$. 

\begin{theorem}\label{th11}
Let $X_{i,n}$, $1 \le i\le n$, $n\ge 1$, be a double array of independent power series rvs defined in \eqref{psdpmf}, and $S_n=\sum_{i=1}^{n}X_{i,n}$. Then
\begin{equation}
d_{TV}(S_n, N_{\lambda})\le \frac{|\lambda - {\mathbb E}S_n|}{\max(1,\sqrt{\lambda})}+\frac{1}{\max(1,\lambda)}\sum_{i=1}^{n}{\mathbb E}X_{i,n}\sum_{k=1}^{\infty}k|p_{i,n}(k)-p_{i,n}^*(k)|,\label{mainboundd}
\end{equation}
where $p_{i,n}(\cdot)$ and $p_{i,n}^*(\cdot)$ are the PSD's given in \eqref{psdpmf} and \eqref{primepsdpmf}, respectively. 
\end{theorem}

\begin{proof}
Replacing $k$ by $S_n$ in \eqref{pooper} and taking expectation, we have
$${\mathbb E}[{\cal A}g(S_n)]=\lambda {\mathbb E}[g(S_n+1)]-{\mathbb E}[S_n g(S_n)].$$
Adding and subtracting ${\mathbb E}S_n {\mathbb E}[g(S_n+1)]$, we get
\begin{align}
{\mathbb E}[{\cal A}g(S_n)]&=(\lambda-{\mathbb E}S_n){\mathbb E}[g(S_n+1)]+{\mathbb E}S_n {\mathbb E}[g(S_n+1)]-{\mathbb E}[S_n g(S_n)]\nonumber\\
&=(\lambda-{\mathbb E}S_n){\mathbb E}[g(S_n+1)]+\sum_{i=1}^{n}{\mathbb E}X_{i,n} {\mathbb E}[g(S_n+1)]-\sum_{i=1}^{n}{\mathbb E}[X_{i,n} g(S_n)].\label{lambda}
\end{align}
Let now
$W_{i,n}=S_n-X_{i,n}$
so that $W_{i,n}$ and $X_{i,n}$ are independent. Adding and subtracting $\sum_{i=1}^{n}{\mathbb E}(X_{i,n}$ $g(W_{i,n}+1))$ in \eqref{lambda}, we have
\begin{align}
{\mathbb E}[{\cal A}g(S_n)]=&(\lambda-{\mathbb E}S_n){\mathbb E}[g(S_n+1)]+\sum_{i=1}^{n}{\mathbb E}X_{i,n} {\mathbb E}[g(S_n+1)-g(W_{i,n}+1)]\nonumber\\
&-\sum_{i=1}^{n}{\mathbb E}[X_{i,n} (g(S_n)-g(W_{i,n}+1))].\label{big}
\end{align}
First, consider the second term from \eqref{big}
\begin{align*}
{\mathbb E}[g(S_n+1)-g(W_{i,n}+1)]&={\mathbb E}[g(W_{i,n}+X_{i,n}+1)-g(W_{i,n}+1)] \\
&=\sum_{k=0}^{\infty}{\mathbb E}[g(W_{i,n}+k+1)-g(W_{i,n}+1)]p_{i,n}(k),
\end{align*}
since $W_{i,n}$ and $X_{i,n}$ are independent. 
Note that
\begin{align}
g(W_{i,n}+k+1)-g(W_{i,n}+1)&=\sum_{j=1}^{k}\Delta g(W_{i,n}+j).\label{delta}
\end{align}
Therefore,
\begin{equation}
{\mathbb E}[g(S_n+1)-g(W_{i,n}+1)]=\sum_{k=1}^{\infty}\sum_{j=1}^{k}{\mathbb E}[\Delta g(W_{i,n}+k)]p_{i,n}(k).\label{1}
\end{equation}
Next, consider the third term of \eqref{big}
\begin{align}
{\mathbb E}[X_{i,n}(g(S_n)-g(W_{i,n}+1))]&={\mathbb E}[X_{i,n}(g(W_{i,n}+X_{i,n})-g(W_{i,n}+1))\nonumber]\\
&=\sum_{k=0}^{\infty}k {\mathbb E}[g(W_{i,n}+k)-g(W_{i,n}+1)]p_{i,n}(k)\nonumber\\
&=\sum_{k=2}^{\infty}\sum_{j=1}^{k-1}k{\mathbb E}[\Delta g(W_{i,n}+j]p_{i,n}(k)\quad\text{(using \eqref{delta})}\nonumber\\
&=\sum_{k=1}^{\infty}\sum_{j=1}^{k}(k+1) {\mathbb E}[\Delta g(W_{i,n}+j)]p_{i,n}(k+1).\label{3}
\end{align}
Substituting \eqref{1} and \eqref{3} in \eqref{big}, we get
\begin{align*}
{\mathbb E}[{\cal A}g(S_n)]&=(\lambda-{\mathbb E}S_n){\mathbb E}[g(S_n+1)]+\sum_{i=1}^{n}{\mathbb E}X_{i,n}\sum_{k=1}^{\infty}\sum_{j=1}^{k}{\mathbb E}[\Delta g(W_{i,n}+j)]p_{i,n}(k)\\
&~~~-\sum_{i=1}^{n}\sum_{k=1}^{\infty}\sum_{j=1}^{k}(k+1){\mathbb E}[\Delta g(W_{i,n}+j)]p_{i,n}(k+1)\\
&=(\lambda-{\mathbb E}S_n){\mathbb E}[g(S_n+1)]\\
&~~~+\sum_{i=1}^{n}\sum_{k=1}^{\infty}{\mathbb E}X_{i,n}\left[p_{i,n}(k)-\frac{(k+1)p_{i,n}(k+1)}{{\mathbb E}X_{i,n}}\right]\sum_{j=1}^{k}{\mathbb E}[\Delta g(W_{i,n}+j)].
\end{align*}
Therefore,
\begin{equation}
|{\mathbb E}[{\cal A}g(S_n)]|\le |\lambda-{\mathbb E}S_n|~\|g\|+\|\Delta g\|\sum_{i=1}^{n}\sum_{k=1}^{\infty}k{\mathbb E}X_{i,n}\left|p_{i,n}(k)-\frac{(k+1)p_{i,n}(k+1)}{{\mathbb E}X_{i,n}}\right|.\label{bbb5}
\end{equation}
Using \eqref{power}, \eqref{bound}, and \eqref{bbb5}, we get
\begin{align}
d_{TV}(S_n, N_{\lambda})&\le \frac{|\lambda-{\mathbb E}S_n|}{\max(1,\sqrt{\lambda})}+\frac{1}{\max(1,\lambda)}\sum_{i=1}^{n}{\mathbb E}X_{i,n}\sum_{k=1}^{\infty}k\left|p_{i,n}(k)-\frac{(k+1)p_{i,n}(k+1)}{{\mathbb E}X_{i,n}}\right|.\label{5}
\end{align}
From \eqref{psdpmf} and \eqref{psdmean}, we have
\begin{equation}
\frac{(k+1)p_{i,n}(k+1)}{{\mathbb E}X_{i,n}}=\frac{(k+1)a_{k+1}\theta_{i,n}^k}{h^\prime(\theta_{i,n})}=p_{i,n}^*(k).\label{last}
\end{equation}
Substituting \eqref{last} in \eqref{5}, we get the required result.
\end{proof}

\begin{corollary}\label{th1}
Let $X_{i,n}$, $1 \le i\le n$, $n \ge 1$, be defined as in Theorem \ref{th11}, and choose $\lambda={\mathbb E}S_n=\sum_{i=1}^{n}{\mathbb E}X_{i,n}=\sum_{i=1}^{n}{\theta_{i,n} h^\prime(\theta_{i,n})}/{h(\theta_{i,n})}$ so that the first moment of $N_{\lambda}$ and $S_n$ match. Then
\begin{equation}
d_{TV}(S_n, N_{\lambda})\le \frac{1}{\max(1,\lambda)}\sum_{i=1}^{n}{\mathbb E}X_{i,n}\sum_{k=1}^{\infty}k|p_{i,n}(k)-p_{i,n}^*(k)|.\label{mainbound}
\end{equation}
\end{corollary}

\noindent
Next we discuss some applications of Theorem \ref{th11} and Corollary \ref{th1}.
\begin{example} {\em
Let $X_{i,n}\sim$ Poi($\lambda_{i.n}$), $1 \le i\le n$ and $n \geq 1$ so that $a_k=1/k!$, $h(\lambda_{i,n})=e^{\lambda_{i,n}}$,
$$p_{i,n}(k)=\frac{e^{-\lambda_{i,n}}\lambda_{i,n}^k}{k!};\quad p_{i,n}^*(k)=\frac{(k+1)a_{k+1}\lambda_{i,n}^k}{h^\prime(\lambda_{i,n})}=\frac{e^{-\lambda_{i,n}}\lambda_{i,n}^k}{k!}=p_{i,n}(k),~k\in{\mathbb Z}_+.$$
Then $S_n=\sum_{i=1}^{n}X_{i,n}\sim \text{Poi}(\lambda_n)$, where $\lambda_n=\sum_{i=1}^{n}\lambda_{i,n}$. From \eqref{mainboundd}, we have
\begin{equation}
d_{TV}(N_{\lambda_n}, N_{\lambda})\le \frac{|\lambda-{\mathbb E}S_n|}{\max(1,\sqrt{\lambda})}=\frac{|\lambda-\lambda_n|}{\max(1,\sqrt{\lambda})}.
\end{equation}
Also, if $\lambda_n\to \lambda$, as $n \to \infty$, then $S_n\stackrel{\cal L}{\to}N_\lambda$.}
\end{example}

\begin{example}\label{ex34} {\em 
Let $X_{i,n}\sim$ Ber($p_{i,n}$), $1 \le i\le n$ so that $p_{i,n}(k)=(1-p_{i,n})^{1-k}p_{i,n}^k$, for $k=0,1$. In that case,  $p_{i,n}^*(k)=1$ for $k=0$ and zero  otherwise. Also, $a_k=1$, for $k=0,1$, and $a_k=0$ for $k \ge 2$, $h(\theta_{i,n})=1+\theta_{i,n}$, where $\theta_{i,n}=p_{i,n}/(1-p_{i,n})$, and ${\mathbb E}X_{i,n}=p_{i,n}$. Then, from \eqref{mainboundd}, we have
$$d_{TV}(S_n, N_{\lambda})\le \frac{|\lambda-\sum_{i=1}^{n}p_{i,n}|}{\max(1,\sqrt{\lambda})}+\frac{\sum_{i=1}^{n}p_{i,n}^2}{\max\left(1,\lambda\right)}.$$
Observe that $d_{TV}(S_n, N_{\lambda})\to 0$ if $\sum_{i=1}^{n}p_{i,n}\to \lambda$, and $\sum_{i=1}^{n}p_{i,n}^2\to 0$, as $n \to \infty$, as proved in Theorem 3 of Wang \cite{WANG1993}. Further, if $\lambda=\sum_{i=1}^{n}p_{i,n}$ then
\begin{align}
d_{TV}(S_n,N_{\lambda})&\le \frac{\sum_{i=1}^{n}p_{i,n}^2}{\max\left(1,\lambda\right)}.\label{ber}
\end{align}
Poisson approximation to the sum of independent Bernoulli rvs has been studied by several authors and some bounds are given below:
\begin{itemize}
\item[(i)] $d_{TV}(S_n,N_{\lambda})\le \sum_{i=1}^{n}p_{i,n}^2$ \quad (Le Cam \cite{CAM})
\item[(ii)] $d_{TV}(S_n,N_{\lambda})\le 1.05 \lambda^{-1} \sum_{i=1}^{n}p_{i,n}^2$ \quad (Kerstan \cite{KJ})
\item[(iii)] $d_{TV}(S_n,N_{\lambda})\le (1-e^{-\lambda}) \lambda^{-1}\sum_{i=1}^{n}p_{i,n}^2$ \quad (Barbour and Hall \cite{BH})
\end{itemize} 
Note that we have used the bound for $\|\Delta g\|$ given in \eqref{bound} to obtain the bound in \eqref{ber}. We will get the Barbour and Hall \cite{BH} bound in (iii) if we use instead the bound $\|\Delta g\|\le (1-e^{-\lambda})/\lambda$ (see (2.6) of Barbour and Hall \cite{BH}).}
\end{example}

\begin{example}\label{exx1} {\em 
	Let $X_{i,n}\sim$ Geo($p_{i,n}$), $1 \le i\le n$. Then $a_k=1$, for all $k \in {\mathbb Z}_+$, $h(\theta_{i,n})=(1-\theta_{i,n})^{-1}$, where $\theta_{i,n}=q_{i,n}=1-p_{i,n}$. Note ${\mathbb E}X_{i,n}=q_{i,n}/p_{i,n}$ and ${\mathbb E}X_{i,n}^*=2q_{i,n}/p_{i,n}$. 
	Also, when $q_{i,n}\le 1/2$, $p_{i,n}^*(k)\ge p_{i,n}(k)$, $k \in {\mathbb Z}_+$. Therefore, from \eqref{mainbound}, we have for $\lambda =\sum_{i=1}^{n}\frac{q_{i,n}}{p_{i,n}}$,
	\begin{align} 
	d_{TV}(S_n, N_{\lambda})&\le\frac{1}{\max(1,\lambda)}\sum_{i=1}^{n}{\mathbb E}X_{i,n}\sum_{k=1}^{\infty}k |p_{i,n}^*(k)-p_{i,n}(k)|\nonumber\\
	&=\frac{1}{\max(1,\lambda)}\sum_{i=1}^{n}\frac{q_{i,n}}{p_{i,n}}\left[\frac{2 q_{i,n}}{p_{i,n}}-\frac{q_{i,n}}{p_{i,n}}\right]=\frac{\sum_{i=1}^{n}\left(\frac{q_{i,n}}{p_{i,n}}\right)^2}{\max\left(1,\lambda\right)}.\label{eqn}
	\end{align}
	Note that if $\max_{1 \le i \le n}q_{i,n}\to 0$, as $n \to \infty$, then $S_{n} \stackrel{\cal L}{\to} N_\lambda$.\\
	Poisson approximation to the sum of independent (identical or non-identical) geometric rvs has been studied by several authors and the bounds are given below:
	\begin{itemize}
		\item[(i)] $d_{TV}(S_n,N_{\lambda})\le (1-e^{-\lambda})\frac{q}{p}$,~~(Barbour \cite{BAD})
		\item[(ii)] $d_{TV}(S_n,N_{\lambda})\le \sum_{i=1}^{n}\frac{q_{i,n}^2}{p_{i,n}^2}\min\left\{1,\frac{1}{\sqrt{2 \lambda e}}\right\}$,~~(Vellaisamy and Upadhye \cite{VPUNS})
		\item[(iii)] $|{\mathbb P}(S_n=k)-{\mathbb P}(X=k)|\le 2 \sum_{i=1}^{n}\left[(1-p_{i,n})^2+\frac{1-p_{i,n}}{p_{i,n}^2}\right]$, ~~(Hung and Giang \cite{HLGT})
		\item[(iv)] $d_{TV}(S_n,N_{\lambda})\le \sum_{i=1}^{n}\min\Big\{\frac{\lambda^{-1}(1-e^{-\lambda})}{p_{i,n}},1\Big\}q_{i,n}^2p_{i,n}^{-1}$, ~~(Teerapabolarn and Wongkasem \cite{TKWP})
	\end{itemize} 
	Observe that the bound given in \eqref{eqn} is comparable to the bound given in (iv) which is an improvement over other bounds. }
\end{example}

\noindent
We next obtain a rather crude bound which would be useful in applications.
\begin{theorem}\label{th2}
Let $X_{i,n}$, $1 \le i\le n$, $n \ge 1$, be defined as in \eqref{psdpmf}, where
$h$ is assumed to be twice differentiable. Let $M_n=\sup_{\theta_{i,n}}[h^\prime(\theta_{i,n})^2+h^{\prime\prime}(\theta_{i,n})h(\theta_{i,n})]$. If $0<M_n<\infty$, then
\begin{equation}
d_{TV}(S_n, N_{\lambda})\le \frac{|\lambda-{\mathbb E}S_n|}{\max(1,\sqrt{\lambda})}+\frac{M_n}{a_0^2  \max(1,\lambda)}\sum_{i=1}^{n}\theta_{i,n}^2.\label{muchcrudebound}
\end{equation}
\end{theorem}

\begin{proof}
Note that the bound given in \eqref{mainboundd} can also be bounded by
\begin{align}
d_{TV}(S_n, N_{\lambda})&\le \frac{|\lambda-{\mathbb E}S_n|}{\max(1,\sqrt{\lambda})}+\frac{1}{\max(1,\lambda)}\sum_{i=1}^{n}{\mathbb E}X_{i,n}\left[{\mathbb E}X_{i,n}+{\mathbb E}X_{i,n}^*\right].\label{crudebound}
\end{align}
Observe that
\begin{align}
{\mathbb E}X_{i,n}^*&=\sum_{k=1}^{\infty}k p_{i,n}^*(k)=\sum_{k=1}^{\infty}\frac{k(k+1)a_{k+1}\theta_{i,n}^k}{h^\prime(\theta_{i,n})}=\frac{\theta_{i,n}}{h^\prime(\theta_{i,n})}\sum_{k=1}^{\infty}k(k-1)a_{k}\theta_{i,n}^{k-2}=\frac{\theta_{i,n}h^{\prime\prime}(\theta_{i,n})}{h^\prime(\theta_{i,n})}.\label{primepsdmean}
\end{align}
Using \eqref{psdmean} and \eqref{primepsdmean} in \eqref{crudebound}, we get
\begin{align}
d_{TV}(S_n, N_{\lambda})&\hspace{-0.05cm}\le\hspace{-0.05cm} \frac{|\lambda-{\mathbb E}S_n|}{\max(1,\sqrt{\lambda})}\hspace{-0.05cm}+\hspace{-0.05cm}\frac{1}{\max(1,\lambda)}\sum_{i=1}^{n}\frac{\theta_{i,n}h^{\prime}(\theta_{i,n})}{h(\theta_{i,n})}\left[\frac{\theta_{i,n}h^{\prime}(\theta_{i,n})}{h(\theta_{i,n})}\hspace{-0.05cm}+\hspace{-0.05cm}\frac{\theta_{i,n}h^{\prime\prime}(\theta_{i,n})}{h^\prime(\theta_{i,n})}\right]\nonumber\\
&=\hspace{-0.05cm}\frac{|\lambda-{\mathbb E}S_n|}{\max(1,\sqrt{\lambda})}\hspace{-0.05cm}+\hspace{-0.05cm}\frac{1}{\max(1,\lambda)}\sum_{i=1}^{n}\left(\frac{\theta_{i,n}}{h(\theta_{i,n})}\right)^2[h^\prime(\theta_{i,n})^2\hspace{-0.05cm}+\hspace{-0.05cm}h^{\prime\prime}(\theta_{i,n})h(\theta_{i,n})]\nonumber\\
&\le\hspace{-0.05cm}\frac{|\lambda-{\mathbb E}S_n|}{\max(1,\sqrt{\lambda})}\hspace{-0.05cm}+\hspace{-0.05cm}\frac{M_n}{\max(1,\lambda)}\sum_{i=1}^{n}\left(\frac{\theta_{i,n}}{h(\theta_{i,n})}\right)^2.\label{crudebound1}
\end{align}
Further, note that
\begin{equation}
h(\theta_{i,n})=\sum_{k=0}^{\infty}a_k \theta_{i,n}^k\ge a_0 \implies \frac{1}{h(\theta_{i,n})^2}\le\frac{1}{a_0^2}.\label{a00}
\end{equation}
Using \eqref{a00} in \eqref{crudebound1}, the result follows.
\end{proof}

\noindent
First, we show that Theorem $3$ of P\'{e}rez-Abreu \cite{APV} follows as a corollary.
\begin{corollary}\label{cor}
Let $X_{i,n}$, $1 \le i\le n$, $n\ge 1$, be defined as in \eqref{psdpmf} with $a_0>0$ and $S_n=\sum_{i=1}^{n}X_{i,n}$. Also, assume
\begin{equation}
\theta_n^*=\max_{1 \le i \le n}\theta_{i,n} \to 0 \quad\text{and}\quad \sum_{i=1}^{n}\theta_{i,n}\to \lambda,~\text{as}~n \to \infty,\label{cond}
\end{equation}
for some $\lambda >0$. Then $S_n \stackrel{\cal L}{\to}N_{\lambda_0}$, where $\lambda_0=\lambda a_1/a_0$.
\end{corollary}
\begin{proof}
It suffices to show 
$d_{TV}(S_n, N_{\lambda_0})\to 0, ~\text{as}~ n \to \infty.$
First, we show that $0<M_n<\infty$. Since $h$ is an increasing function and $\theta_n^* \to 0$, as $n \to \infty$, we have $a_0\le h(\theta_{i,n})\le h(\theta_n^*)\to a_0$, showing that $h(\theta_{i,n})\to a_0$, as $\theta_n^*\to 0$. Using the similar argument, since $h^\prime$ and $h^{\prime\prime}$ are also increasing functions, we have $h^\prime(\theta_{i,n})\to a_1$ and $h^{\prime\prime}(\theta_{i,n})\to 2a_2$, as $\theta_n^* \to 0$. Also,
\begin{align}
M_n&=\sup_{\theta_{i,n}}[h^\prime(\theta_{i,n})^2+h^{\prime\prime}(\theta_{i,n})h(\theta_{i,n})]\le \sup_{\theta_{n}^*}[h^\prime(\theta_{n}^*)^2+h^{\prime\prime}(\theta_{n}^*)h(\theta_{n}^*)]\to (a_1^2+2a_0a_2),\label{cond2}
\end{align}
as $\theta_n^* \to 0$. Hence, $0<M_n < \infty$.

\noindent Applying Theorem \ref{th2}, we have from  \eqref{muchcrudebound}, 
\begin{align}
d_{TV}(S_n, N_{\lambda_0})&\le \frac{|\lambda_0-{\mathbb E}S_n|}{\max(1,\sqrt{\lambda_0})}+\frac{M_n \theta_n^*}{a_0^2  \max(1,\lambda_0)}\sum_{i=1}^{n}\theta_{i,n}.\label{tendto0}
\end{align}
 Note also that
\begin{equation}
{\mathbb E}S_n=\sum_{i=1}^{n}\frac{\theta_{i,n}h^\prime(\theta_{i,n})}{h(\theta_{i,n})}\to \frac{\lambda a_1}{a_0}=\lambda_0,\label{sccond}
\end{equation}
under the conditions \eqref{cond}. The result now follows from \eqref{tendto0},
 \eqref{sccond} and the assumptions.
\end{proof}

\begin{example} {\em 
	Let $X_{i,n}\sim$ Geo($p_{i,n}$), $1 \le i\le n$. Then $a_k=1$, for all $k \in {\mathbb Z}_+$, $h(\theta_{i,n})=(1-\theta_{i,n})^{-1}$, where $\theta_{i,n}=q_{i,n}=1-p_{i,n}$. If $\max_{1 \le i \le n}q_{i,n}\to 0$ and $\sum_{i=1}^{n}q_{i,n} \to \lambda$, as $n \to \infty$, then by Corollary \ref{cor}, $S_{n} \stackrel{\cal L}{\to} N_{\lambda}$, since $a_0=a_1=1$.}
	
\end{example}

\noindent
The next example concerns the Poisson approximation to the sum of logarithmic series rvs, which  has not been considered  in the literature.

\begin{example}\label{ex23} {\em 
Let $Y_{i,n}$, $1 \le i\le n$, $n \ge 1$, follow logarithmic series distribution with
$${\mathbb P}(Y_{i,n}=k)=-\frac{\theta_{i,n}^k}{k\ln(1-\theta_{i,n})},\quad 0<\theta_{i,n}<1,~ k=1,2,3,\dotsc.$$
Further, let $X_{i,n}=Y_{i,n}-1$. Then 
\begin{equation}
{\mathbb P}(X_{i,n}=k)=p_{i,n}(k)=-\frac{\theta_{i,n}^{k+1}}{(k+1)\ln(1-\theta_{i,n})},\quad k\in{\mathbb Z}_+.\label{vv1}
\end{equation}
Our interest is to obtain the error bound for the Poisson approximation to $S_n=\sum_{i=1}^{n}X_{i,n}$. Here, $a_k=1/(k+1)$, $h(\theta_{i,n})=-\ln(1-\theta_{i,n})/\theta_{i,n}$ and therefore,
\begin{equation}
p_{i,n}^*(k)=\frac{(k+1)a_{k+1}\theta_{i,n}^k}{h^\prime(\theta_{i,n})}=\frac{(k+1)\theta_{i,n}^{k+1}}{(k+2)\left(\frac{1}{1-\theta_{i,n}}+\frac{\ln(1-\theta_{i,n})}{\theta_{i,n}}\right)},\quad k\in{\mathbb Z}_+.\label{vv2}
\end{equation}
From \eqref{mainbound}, we have
\begin{equation}
d_{TV}(S_n, N_{\lambda})\le \frac{1}{\max(1,\lambda)}\sum_{i=1}^{n}{\mathbb E}X_{i,n}\sum_{k=1}^{\infty}k|p_{i,n}(k)-p_{i,n}^*(k)|,\label{vvv}
\end{equation}
where ${\mathbb E}(X_{i,n})={\mathbb E}(Y_{i,n})-1=-1-\theta_{i,n}/[(1-\theta_{i,n})\ln(1-\theta_{i,n})]$, $\lambda=\sum_{i=1}^{n}{\mathbb E}X_{i,n}$, and $p_{i,n}(k)$ and $p_{i,n}^*(k)$ defined as in \eqref{vv1} and \eqref{vv2}, respectively. Also, if $\max_{1 \le i \le n}\theta_{i,n}\to 0$ and $\sum_{j=1}^{n}\theta_{j,n} \to \lambda$, as $n \to \infty$, then, by Corollary \ref{cor}, $S_{n} \stackrel{\cal L}{\to} N_{\lambda/2}$, since $a_0=1$ and $a_1=1/2$. }
\end{example}

\subsection{The Identical Distributions Case}
Let $X_{i}$, $1 \le i\le n~ (n \geq 1)$, be a sequence of independent and identical rvs having the PSD with  pmf 
\begin{equation}
p_{n}(k)={\mathbb P}(X_{i}=k)=\frac{a_k \theta_{n}^k}{h(\theta_{n})},~k\in {\mathbb Z}_+,\label{psdpmf1}
\end{equation}
where $a_k\ge 0$, $\theta_{n}>0$, and $h(\theta_{n})=\displaystyle{\sum_{k=0}^{\infty}}a_k \theta_{n}^k$. Also, let $X_{i}^*$, $1 \le i\le n$, be a sequence of independent rvs having PSD corresponding to $h^\prime (\theta_{n})$, so that the pmf of $X_{i}^*$ is given by
\begin{equation}
p_{n}^*(k)={\mathbb P}(X_{i}^*=k)=\frac{(k+1)a_{k+1}\theta_{n}^k}{h^\prime(\theta_{n})},~k \in {\mathbb Z}_+.\label{primepsdpmf1}
\end{equation}

\begin{theorem}\label{thiid1}
Let $X_{i}$ and $X_{i}^*$, $1 \le i\le n$ be a sequence of rvs with pmf's $p_{n}(\cdot)$ and $p_{n}^*(\cdot),$  defined in \eqref{psdpmf1} and \eqref{primepsdpmf1}, respectively. Let $S_n=\sum_{i=1}^{n}X_i$, and $\mu_n={\mathbb E}S_n=n\theta_n h^\prime(\theta_n)/h(\theta_n)$. Then, from Theorem \ref{th11}, we have
\begin{align}
d_{TV}(S_n, N_{\lambda})&\le \frac{\left|\lambda-\mu_n\right|}{\max(1,\sqrt{\lambda})}+\frac{\mu_n}{\max(1,\lambda)}\sum_{k=1}^{\infty}k|p_{n}(k)-p_{n}^*(k)|,\label{mainboundiid}
\end{align}
Also, from Theorem  \ref{th2}, a crude bound is
\begin{align}
d_{TV}(S_n, N_{\lambda})&\le \frac{\left|\lambda-\mu_n\right|}{\max(1,\sqrt{\lambda})}+\frac{n \theta_n^2[h^\prime(\theta_{n})^2+h^{\prime\prime}(\theta_{n})h(\theta_{n})]}{a_0^2  \max(1,\lambda)}.\label{muchcrudebound5}
\end{align}
\end{theorem}

\noindent
The following  result is Theorem $2$ of P\'{e}rez-Abreu \cite{APV}.
\begin{corollary}\label{cor1}
Let the conditions of Theorem \ref{thiid1} hold and $n \theta_{n}\to \lambda,~\text{as}~n \to \infty$, for some $\lambda >0$. Then $S_n \stackrel{\cal L}{\to}N_{\lambda_0}$, where $\lambda_0=\lambda a_1/a_0$ with $a_0>0$.
\end{corollary}

\noindent
Next, we discuss two results, namely, binomial and negative binomial convergence to Poisson as application of Theorem \ref{thiid1}.

\begin{corollary}[Poisson Approximation to Binomial Distribution]\label{ex2}
Let $X_{i}\sim$ Ber$(p_n)$, for $1 \le i\le n$, and $S_n=\sum_{i=1}^{n}X_i$. If $p_n \to 0$ and $n p_n \to\lambda$, as $n \to \infty$, then $S_n\stackrel{\cal L}{\to}N_\lambda$.
\end{corollary}

\begin{proof}
When $X_{i}\sim$ Ber$(p_n)$, $1 \le i\le n$, we have $a_k=1$, for $k=0,1$, and $a_k=0$, for $k\ge 2$. Also, $h(\theta_{n})=1+\theta_{n},$ where $\theta_{n}=p_n/(1-p_n)$ and $S_n\sim \text{Bi} (n,p_n)$. Note that $h^\prime(\theta_{n})=1~\text{and}~h^{\prime\prime}(\theta_{n})=0.$ Hence, from \eqref{muchcrudebound5},
\begin{align}
d_{TV}(S_n, N_{\lambda})&\le\frac{|\lambda-np_n|}{\max(1,\sqrt{\lambda})}+\frac{np_n^2}{\max(1,\lambda)(1-p_n)^2}.
\end{align}
The bound given above goes to zero if $p_n\to 0$ and $np_n\to \lambda$, as $n \to \infty$. This proves the result.
\end{proof}

\begin{corollary}[Poisson Approximation to Negative Binomial Distribution]\label{ex3}
Let $X_{i}\sim$ Geo$(p_n)$, for $1 \le i\le n$, and $S_n=\sum_{i=1}^{n}X_i$. If $p_n\to 1$ and $n(1-p_n)\to \lambda$, as $n \to \infty$, then $S_n\stackrel{\cal L}{\to}N_\lambda$.
\end{corollary}

\begin{proof}
Since $X_{i}\sim Geo(p_n)$, $1 \le i\le n$, we have $a_k=1$, for all $k \in {\mathbb Z}_+$. Also, $h(\theta_{n})=(1-\theta_{n})^{-1},$ where $\theta_{n}=1-p_n$ and $S_n\sim $ NB($n,p_n$). Note that $h^\prime(\theta_{n})=(1-\theta_{n})^{-2}=p_n^{-2}$ and $h^{\prime\prime}(\theta_{n})=2(1-\theta_{n})^{-3}=2p_n^{-3}$. Hence, from \eqref{muchcrudebound5}, we have
\begin{align*}
d_{TV}(S_n, N_{\lambda})&\le \frac{|\lambda-\frac{n(1-p_n)}{p_n}|}{\max(1,\sqrt{\lambda})}+\frac{3n(1-p_n)^2}{\max(1,\lambda)p_n^4}\to 0,
\end{align*}
if $p_n\to 1$ and $n(1-p_n)\to \lambda$, as $n \to \infty$. This proves the result.
\end{proof}

\begin{example}
Consider the similar argument discussed in Example \ref{ex23} for logarithmic series distribution under identical setup. Then, we have $a_k=1/(k+1)$ and $h(\theta_n)=-\ln(1-\theta_n)/\theta_n$. Also, let $n \theta_n \to \lambda$, as $n \to \infty$. Then by Corollary \ref{cor1}, $S_{n} \to N_{\lambda/2}$, since $a_0=1$ and $a_1=1/2$.
\end{example}

\section{Comparison of Poisson and Negative Binomial Bounds}\label{CPNB}
In this section, we present the results for negative binomial approximation derived by Vellaisamy {\em et al.} \cite{VUC} for the PSD's. Also, we compare the bounds with the Poisson approximation bounds through some relevant examples. Let $M_{r,p}$ denote the negative binomial rv with parameter $r>0$ and $p=1-q\in (0,1)$ with
$${\mathbb P}(M_{r,p}=k)=\binom{r+k-1}{k}p^r q^k,\quad k \in{\mathbb Z}_+.$$
Also, let $S_n=\sum_{i=1}^{n}X_{i,n}$, where $X_i$'s are defined in \eqref{psdpmf} such that
\begin{equation}
{\mathbb E}S_n=\displaystyle{\sum_{i=1}^{n}}{\mathbb E}X_{i,n}=\sum_{i=1}^{n}\frac{\theta_{i,n} h^\prime(\theta_{i,n})}{h(\theta_{i,n})}=\frac{rq}{p}.\label{meancond}
\end{equation}
The condition \eqref{meancond} implies ${\mathbb E}(S_n)={\mathbb E}(M_{r,p})$, the first moment matching. Then, from Theorem $3.1$ of Vellaisamy {\em et al.} \cite{VUC}, the one-parameter approximation bound is
\begin{align*}
d_{TV}(S_n, M_{r,p})&\le \frac{1}{rp}\sum_{i=1}^{n}\sum_{k=1}^{\infty}k|p{\mathbb E}X_{i,n}p_{i,n}(k)+qkp_{i,n}(k)-(k+1)p_{i,n}(k+1)|\\
&=\frac{1}{rp}\sum_{i=1}^{n}{\mathbb E}X_{i,n}\sum_{k=1}^{\infty}k\left|p~p_{i,n}(k)+q\frac{kp_{i,n}(k)}{{\mathbb E}X_{i,n}}-\frac{(k+1)p_{i,n}(k+1)}{{\mathbb E}X_{i,n}}\right|.
\end{align*} 
Using \eqref{last}, we have
\begin{equation}
d_{TV}(S_n, M_{r,p})\le \frac{1}{rq}\sum_{i=1}^{n}{\mathbb E}X_{i,n}\sum_{k=1}^{\infty}k|p~p_{i,n}(k)+q~p_{i,n}^*(k-1)-p_{i,n}^*(k)|.\label{mainboundNB}
\end{equation}
Furthermore, let
$$\tau:=2\max_{1\le i \le n}d_{TV}(W_{i,n},W_{i,n}+1)=\max_{1\le i\le n}\sum_{k=1}^{\infty}|{\mathbb P}(W_{i,n}=k)-{\mathbb P}(W_{i,n}=k-1)|,$$
where $W_{i,n}= (S_n-X_{i,n})$. Choose now  $r$ and $p$ such that
\begin{equation}
r=\frac{({\mathbb E}S_n)^2}{\mathrm{Var}(S_n)-{\mathbb E}S_n}\quad \text{and} \quad p=\frac{{\mathbb E}S_n}{\mathrm{Var}(S_n)}.\label{meanvarcond}
\end{equation}
Then, from Theorem $4.1$ of Vellaisamy {\em et al.} \cite{VUC}, two-parameter approximation bound is
\begin{align*}
d_{TV}(S_n, M_{r,p})\le \frac{\tau}{rq}\sum_{i=1}^{n}\sum_{k=1}^{\infty}k\left(\frac{k-1}{2}+{\mathbb E}X_{i,n}\right)|p{\mathbb E}X_{i,n}p_{i,n}(k)+qk p_{i,n}(k)-(k+1)p_{i,n}(k+1)|.
\end{align*}
Using \eqref{last}, we have
\begin{equation}
d_{TV}(S_n, M_{r,p})\le \frac{\tau}{rq}\sum_{i=1}^{n}{\mathbb E}X_{i,n}\sum_{k=1}^{\infty}k\left(\frac{k-1}{2}+{\mathbb E}X_{i,n}\right)|p~p_{i,n}(k)+q~p_{i,n}^*(k-1)-p_{i,n}^*(k)|.\label{mainboundNB1}
\end{equation}
Also, from Remark $4.1$ of Vellaisamy {\em et al.} \cite{VUC},  $\tau \le \sqrt{\frac{2}{\pi}}\left(\frac{1}{4}+\sum_{i=1}^{n}\tau_i-\tau^*\right)^{-1/2}$, where $\tau_i=\min\{\frac{1}{2}, 1-d_{TV}(X_{i,n},X_{i,n}+1)\}$ and $\tau^*=\max_{1 \le i \le n}\tau_i$.\\
When $X_{i,n}\sim $ Ber($p_{i,n}$), $1 \le i\le n$, and fix $r=n$ then, using Example \ref{ex34} and \eqref{mainboundNB}, we have
\begin{equation}
d_{TV}(S_n,M_{n,p})\le \frac{2 \sum_{i=1}^{n}p_{i,n}^2}{\sum_{i=1}^{n}p_{i,n}},\label{qwqw}
\end{equation}
where $p=1/\left(1+(1/n)\sum_{i=1}^{n}p_{i,n}\right)$. Observe that the Poisson approximation (see \eqref{ber}) in this case is better than the negative binomial approximation. Also, the bound given in \eqref{mainboundNB1} is not valid as the condition \eqref{meanvarcond} is not a valid choice of parameters. 

\noindent
When $X_{i,n}\sim$ Geo($p_{i,n}$), $1 \le i\le n$, and for $q_{i,n}\le 1/2$, we have from (14) and (19) of Vellaisamy {\em et al.} \cite{VUC} that
\begin{equation}
d_{TV}(S_n,M_{r,p})\le \frac{1}{rq}\sum_{i=1}^{n}|p-p_{i,n}|\frac{q_{i,n}}{p_{i,n}^2}\label{eqn1}
\end{equation} 
and
\begin{equation}
d_{TV}(S_n,M_{r,p})\le 3 \sqrt{\frac{2}{\pi}}\left(\sum_{i=1}^{n}q_{i,n}-\frac{1}{4}\right)^{-1/2}\left(\sum_{i=1}^{n}\frac{q_{i,n}}{p_{i,n}}\right)^{-1}\sum_{i=1}^{n}\left|\frac{1}{p_i}-\frac{1}{p}\right|\left(\frac{q_{i,n}}{p_{i,n}}\right)^2,\label{eqn2}
\end{equation}
respectively. Now, we give a numerical comparison of the above bounds with the Poisson approximation bound. Let us choose the values of $q_{i,n}=1-p_{i,n}$ for various values of $i$ as follows:

\begin{table}[H]
  \centering
  \caption{The values of $q_{i,n}$ for numerical computations.}
   \label{tab1}
\resizebox{0.95\textwidth}{!}{%
\begin{tabular}{cccccccccccc}
\toprule
$i$       & 1-10 & 11-20 & 21-50 & 51-100 & 101-150 &  151-200 & 201-250 & 251-300 & 301-400 & 401-500\\
$q_{i,n}$ & 0.20 & 0.18   & 0.16    & 0.14    & 0.12    & 0.10    & 0.08    & 0.06    & 0.04    & 0.02\\
\bottomrule
\end{tabular}}
\end{table}

\noindent
Also, from \eqref{eqn}, \eqref{eqn1} and \eqref{eqn2}, we have the following numerical computation of the bounds.

\begin{table}[H]
  \centering
  \caption{Comparison of bounds.}
  \resizebox{0.95\textwidth}{!}{\begin{tabular}{lccccccccccc}
 \toprule
\multirow{2}{*}{$n$} & Poisson Approximation & Negative Binomial Approximation & Negative Binomial Approximation\\
                     & (from \eqref{eqn})& (one moment matching, from \eqref{eqn1}) & (two moments matching, from \eqref{eqn2})\\
\midrule
\vspace{0.05cm}
10 & 0.25 & 0 & 0 \\
\vspace{0.05cm}
20 & 0.2357460 & 0.0152439 & 0.0045271 \\
\vspace{0.05cm}
50 & 0.2108950 & 0.0221529 & 0.0040439 \\
\vspace{0.05cm}
100& 0.1897860 & 0.0242177 & 0.0029142 \\
\vspace{0.05cm}
150& 0.1754270 & 0.0279765 & 0.0028037 \\
\vspace{0.05cm}
200& 0.1638720 & 0.0329378 & 0.0025511 \\
\vspace{0.05cm}
250& 0.1543910 & 0.0379188 & 0.0025933 \\
\vspace{0.05cm}
300& 0.1468760 & 0.0436179 & 0.0025801 \\
\vspace{0.05cm}
400& 0.1365930 & 0.0531102 & 0.0024826 \\
500& 0.1312850 & 0.0612465 & 0.0024836 \\
\bottomrule
\end{tabular}}
\end{table}

\noindent
Note that, for $n \le 10, p_{i,n}=0.8 $ and so the bounds for negative binomial approximation are zero as sum of iid geometric is negative binomial  From the table, we see that negative binomial approximation is better than Poisson approximation. Also, the bound obtained by matching first two moments is better than the bound obtained by matching one moment, as expected.

\begin{example} {\em
Consider again Example \ref{ex23} for logarithmic series distribution. Note that
\begin{align*}
p_{i,n}(k-1)-p_{i,n}(k)&=-\frac{\theta_{i,n}^{k}}{k(k+1)\ln(1-\theta_{i,n})}\left[k(1-\theta_{i,n})+1\right]\ge 0.
\end{align*}
Therefore,
\begin{align*}
d_{TV}(X_{i,n},X_{i,n}+1)&= \frac{1}{2}\sum_{k=0}^{\infty}|p_{i,n}(k-1)-p_{i,n}(k)|=\frac{1}{2}\left[p_{i,n}(0)+\sum_{k=1}^{\infty}[p_{i,n}(k-1)-p_{i,n}(k)]\right]\\
&=\frac{1}{2}\left[p_{i,n}(0)+p_{i,n}(0)\right]=p_{i,n}(0)=-\frac{\theta_{i,n}}{\ln(1-\theta_{i,n})}.
\end{align*}
From \eqref{mainboundNB} and \eqref{mainboundNB1}, we have
\begin{equation}
d_{TV}(S_n, M_{r,p})\le \frac{1}{rq}\sum_{i=1}^{n}{\mathbb E}X_{i,n}\sum_{k=1}^{\infty}k|p~p_{i,n}(k)+q~p_{i,n}^*(k-1)-p_{i,n}^*(k)|\label{vvv1}
\end{equation}
and
\begin{equation}
d_{TV}(S_n, M_{r,p})\le \frac{\tau}{rq}\sum_{i=1}^{n}{\mathbb E}X_{i,n}\sum_{k=1}^{\infty}k\left(\frac{k-1}{2}+{\mathbb E}X_{i,n}\right)|p~p_{i,n}(k)+q~p_{i,n}^*(k-1)-p_{i,n}^*(k)|,\label{vvv2}
\end{equation}
where $\tau \le \sqrt{\frac{2}{\pi}}\left(\frac{1}{4}+\sum_{i=1}^{n}\tau_i-\tau^*\right)^{-1/2},$  $\tau_i=\min\{\frac{1}{2}, 1+(\theta_{i,n}/\ln(1-\theta_{i,n}))\}$ and $\tau^*=\max_{1 \le i \le n}\tau_i$, ${\mathbb E}(X_{i,n})={\mathbb E}(Y_{i,n})-1=-1-\theta_{i,n}/[(1-\theta_{i,n})\ln(1-\theta_{i,n})]$, and $p_{i,n}(k)$ and $p_{i,n}^*(k)$ defined as in \eqref{vv1} and \eqref{vv2}, respectively. Also, for \eqref{vvv1} and \eqref{vvv2}, $r$ and $p$ can be evaluated using the conditions \eqref{meancond} and \eqref{meanvarcond}, respectively.

\noindent
From \eqref{vvv} and \eqref{vvv1} with $r=n/5$ and $p=1/\left(1+(5/n)\sum_{i=1}^{n}(-1-\theta_{i,n}/[(1-\theta_{i,n})\ln(1-\theta_{i,n})])\right)$, and \eqref{vvv2}, we have the following numerical computation of the bounds with $\theta_{i,n}=q_{i,n}$ given in Table \ref{tab1}. It is difficult to compute the bounds in a compact form. So, the bounds are computed by using Mathematica Software, version 11.3. 
\begin{table}[H]
  \centering
  \caption{Comparison of bounds.}
  \resizebox{0.95\textwidth}{!}{\begin{tabular}{lccccccccccc}
 \toprule
\multirow{2}{*}{$n$} & Poisson Approximation & Negative Binomial Approximation & Negative Binomial Approximation\\
                     & (from \eqref{vvv})& (one moment matching, from \eqref{vvv1}) & (two moments matching, from \eqref{vvv2})\\
\midrule
\vspace{0.05cm}
10 & 0.2068330 & 0.3949420 & 0.0033845 \\
\vspace{0.05cm}
20 & 0.1950790 & 0.3711290 & 0.0033441 \\
\vspace{0.05cm}
50 & 0.1745720 & 0.3293270 & 0.0028029 \\
\vspace{0.05cm}
100& 0.1571360 & 0.2931890 & 0.0022248 \\
\vspace{0.05cm}
150& 0.1452490 & 0.2661830 & 0.0019849 \\
\vspace{0.05cm}
200& 0.1356570 & 0.2411430 & 0.0019339 \\
\vspace{0.05cm}
250& 0.1277630 & 0.2165510 & 0.0019276 \\
\vspace{0.05cm}
300& 0.1214880 & 0.1917620 & 0.0019011 \\
\vspace{0.05cm}
400& 0.1128780 & 0.1479230 & 0.0018827 \\
500& 0.1084220 & 0.1086410 & 0.0018696 \\
\bottomrule
\end{tabular}}
\end{table}

\noindent
From the above table, it is clear  that the Poisson approximation is better than negative binomial approximation with one moment matching. However, the negative binomial approximation with two moments matching is better than  the Poisson approximation with one moment matching.}
\end{example}

\singlespacing
\small
\bibliographystyle{PV}
\bibliography{PA2PSDBib}

\end{document}